%% file: mainrnf.tex
\documentclass[10pt]{article}
\input{pack}

\input{config}    

\begin{document}
\input{hearnf}      
\input{absrnf}       
\newline
\newline
{\bf Keywords:} Fourier analysis, iterative method, preconditioners, eigenvalues, eigenvectors, GMRES
\newline
\newline
{\bf AMS classification:} 65F10, 65F15, 65F08, 65T99  
\input{intrnf2}
\input{notations}
\input{anarnf}
\input{gmgrnf}
\input{numrnf}
\input{conrnf}

\input{bibrnf}       
\end{document}

%% file: pack.tex
\usepackage{helvet}
\usepackage[all]{xy}
\usepackage{amssymb, amsthm}
\usepackage{epsfig}

\usepackage{hyperref}

\usepackage{amsmath}
\usepackage{mathrsfs}
\usepackage{algorithmic}
\usepackage{algorithm}

\usepackage{rotating}
\usepackage{multirow}
\usepackage{booktabs}

%% file: config.tex
\graphicspath{{./figs/}}

\newtheorem{theorem}{Theorem}
\newtheorem{lemma}{Lemma}

\newtheorem{definition}{Definition}
\newtheorem{corollary}{Corollary}
\newtheorem{notation}{Notation}

\newenvironment{packed_enum}{
\begin{enumerate}
\topsep=0pt plus 2pt minus 4pt
  \setlength{\itemsep}{1pt}
  \setlength{\parskip}{0pt}
  \setlength{\parsep}{0pt}
}{\end{enumerate}}

\def\iluz{{ILU(0)}} 
\def\milu{{MILU}}

%% file: hearnf.tex
\begin{center}
  \textbf{\LARGE Analysis of hierarchical SSOR for three dimensional isotropic model problem}\vspace*{3ex}\\
  Pawan Kumar\footnote{Some part of this work was done when the author was supported by Fonds de la recherche scientifique (FNRS)(Ref: 2011/V 6/5/004-IB/CS-15) and the renewed contract (Ref: 2011/V 6/5/004-IB/CS-9980) at Universit\'e Libre de Brussels, Belgique and the remaining work was done at KU Leuven, Leuven, Belgium}\\
  Department of computer science \\
  KU Leuven \\
  Leuven, Belgium \\
  \textsf{pawan.kumar@cs.kuleuven.be}\vspace*{0.5ex}\\
\end{center}

%% file: absrnf.tex
\section*{Abstract}
In this paper, we study a hierarchical SSOR (HSSOR) method which could
be used as a standalone method or as a smoother for a two-grid
method. It is found that the method leads to faster convergence
compared to more costly incomplete LU (ILU(0)) with no fill-in, the
SSOR, and the Block SSOR method. Moreover, for a two-grid method,
numerical experiments suggests that HSSOR can be a better replacement
for SSOR smoother both having no storage requirements and have no
construction costs. Using Fourier analysis, expressions for the
eigenvalues and the condition number of HSSOR preconditioned problem
is derived for the three-dimensional isotropic model problem.

%% file: intrnf2.tex
\section{Introduction}
Fourier analysis has been an indispensable tool in understanding
existing algorithms and designing newer algorithms in computational
science. Similar techniques have been used extensively for analyzing
the iterative methods for solving linear systems of the form,
\begin{align} \label{E:linearSystem} Ax=b
\end{align}
which lies at the heart of plenty of scientific simulations ranging
from computational fluid dynamics to sparse numerical optimization.

The matrix $A$ may be ill-conditioned and some preconditioning is
often necessary during an iterative procedure. A possible
preconditioned linear system is a transformation of the linear system
\eqref{E:linearSystem} to $B^{-1}Ax=B^{-1}b$ where $B$ is an
approximation to $A$.  The basic linear iteration for solving the
preconditioned linear system \eqref{E:linearSystem} above is given as
follows
\[ x^{n+1} = x^n + B^{-1}(b-Ax^n).\] The iteration is stopped as soon
as the residual $b-Ax^n$ is small enough under suitable norm. However,
the fixed point iteration shown above is very slow, and it is usually
replaced by a more sophisticated and relatively robust Krylov subspace
based methods \cite{saad96} where the solution is improved with the
help of Krylov space $\{ \, b, B^{-1}Ab, \dots, (B^{-1}A)^{r}b \,
\}$. When the preconditioned operator $B^{-1}A$ is symmetric positive
definite, a popular choice is the conjugate gradient method
\cite{saad96}.  The convergence of the conjugate gradient method
depends on the condition number of $B^{-1}A$ which in this case is
simply the ratio of its largest and smallest eigenvalue.

Most of the properties such as eigenvalue bounds and the condition
number of the classical preconditioners including Jacobi,
Gauss-Siedel, SSOR, alternating direction method (ADI), sparse
incomplete LU with no fill, i.e., ILU(0), and modified ILU (MILU)
which were earlier obtained by difficult analysis, were obtained
easily and elegantly via Fourier analysis by Chan and Elman in
\cite{ton}. For a two dimensional model problem, Fourier analysis was
used to determine the condition number of block MILU for a hyperbolic
model problem \cite{ott}.  Using a similar approach, the condition
number of {\iluz} and {\milu} for a three dimensional anisotropic
model problem is derived in \cite{don}. On the other hand, for
two-grid and multigrid methods, Fourier analysis has been used to
understand the action of the smoother in damping the error
\cite{wie}. Recently, the author used similar analysis to determine
the condition number of a filtering preconditioner \cite{niu}.

A preconditioner known as RNF(0,0) was first introduced in
\cite{kum}. It is a modified form of the nested factorization method
introduced in \cite{appleyardChesire}.  For simplicity and its
resemblance to SSOR method, we shall call this method
\emph{Hierarchical} SSOR (HSSOR). We call it hierarchical because the
preconditioner is built hierarchically using the SSOR preconditioner
on lower dimensions thereby using the structure of the matrix. Like
point-wise SSOR method, the HSSOR method has no storage requirements
and has no construction cost. On the other hand, unlike block SSOR
method where explicit factorization of 2D blocks are required, for
HSSOR, no explicit factorization of any lower dimensional block is
ever performed. In fact, the 2D plane blocks themselves have SSOR
factors. The convergence of the method is faster compared to ILU(0) as
shown in this paper. Our empirical results suggests that HSSOR is a
better replacement for SSOR (or Gauss-Siedel) smoother which is widely
employed in the two-grid schemes. For symmetric positive definite
coefficient matrix, it was proved in \cite{kum} that the HSSOR method
(called RNF(0,0) in \cite{kum}) is convergent. In this paper, we
derive the precise expression for eigenvalue and the condition number
of the HSSOR preconditioned matrix for the three-dimensional isotropic
model problem.  Since our analysis uses the same model problem used in
\cite{don}, our condition number estimate can be compared to that of
ILU(0).

The rest of this paper is organized as follows. In section 2, we
introduce some notations, the model problem that we use, and finally
we discuss the HSSOR method. In section 3, the Fourier eigenvalues and
eigenvectors are derived, and several properties including the
condition number estimate is presented. The numerical experiments and
comparisons are presented in section 4. Finally, section 5 concludes
the paper.

%% file: notations.tex
\section{Model problem and the preconditioner} 
The model problem is the following 
three-dimensional anisotropic equation:
\begin{align}
 -(l_{1}u_{xx} + l_{2}u_{yy} + l_{3}u_{zz}) = r \label{pde}
\end{align}
defined on a unit cube $\Omega = \{\, 0 \leq x,y,z \leq 1 \, \}$, with 
$l_{1}, l_{2}, l_{3} \geq 0$, and with the periodic boundary conditions as follows 
\begin{align} \label{perBdry}
 u(x,y,0) = u(x,y,1), \quad u(x,0,z) = u(x,1,z),  \quad \text{and}~u(0,y,z) = u(1,y,z).
\end{align}
The discretization scheme considered in the interior of the domain is the 
second order finite differences on a uniform $n \times n \times n$ 
grid, with mesh size $h = 1/(n+1)$ along $x$, $y$, and $z$ directions.
Here we shall use the notation $h$ to denote the mesh size for the periodic case. With this discretization, 
we get a system of equation
\begin{eqnarray}
 A u = b \label{eqnA}.
\end{eqnarray}
It is useful to express the matrix $A$ arising from the periodic
boundary conditions using the notation of circulant matrices and the 
Kronecker products. We introduce these notations as follows.
\begin{definition} Let $C$ be a matrix of size $pq \times pq$. 
We call $C$ a \textbf{block circulant} matrix if it has the following form\\
\small
\begin{displaymath}
C = Bcirc_{p}(C_{0},C_{p-1},\cdots,C_{2},C_{1}) = \left( \begin{array}{lllll}
C_{0} & C_{p-1} & \cdots & C_{2} & C_{1} \\ 
C_{1} & C_{0} & C_{p-1} & \ddots & \vdots \\ 
\vdots & C_{1} & C_{0} & \ddots & \vdots  \\ 
C_{p-2} &   & \ddots & \ddots & C_{p-1}  \\ 
C_{p-1} & C_{p-2} & \cdots & C_{1} & C_{0}  \\ 
\end{array} \right)_{pq \times pq},
\end{displaymath}
\normalsize
\end{definition}
where each of the blocks $C_{i}$ are matrices of size $q \times q$ each.
We observe that a block circulant matrix is completely specified by 
a block row. However if $q=1$, then we call it \textbf{circulant matrix}
and denote it by $circ_{p}(C_{0},C_{p-1},\cdots,C_{2},C_{1})$.
\begin{notation}
Further, for \textbf{block circulant tridiagonal} matrices we introduce the 
following notation
\small
\begin{displaymath}
Bctrid_{p}(C_{2}, C_{0}, C_{1}) =  \left(\begin{array}{lllll}
C_{0} 	& C_{1} 	& 	 	&  		& C_{2} \\ 
C_{2}	& \ddots 	& \ddots 	& 		&  	 \\ 
 	& \ddots 	& \ddots 	& \ddots	& 	 \\ 
   	&       	& \ddots      	& C_{0} 	& C_{1} \\ 
C_{1} &  		&  		& C_{2} 	& C_{0} 
\end{array}\right)_{pq \times pq},
\end{displaymath}
\normalsize
\end{notation}
where each of the blocks $C_{i}$ are matrices of size $q \times q$ each.
However if $q=1$, then we denote it by $ctrid_{p}(C_{2}, C_{0}, C_{1})$.
\begin{notation}
For \textbf{block tridiagonal matrix} with constant block bands we introduce the following
notation
\small
\begin{displaymath}
Btrid_{p}(F_{2}, F_{0}, F_{1}) =  \left(\begin{array}{lllll}
F_{0} 	& F_{1} 	& 	 	&  		& 	 \\ 
F_{2}	& \ddots 	& \ddots 	& 		&  	 \\ 
 	& \ddots 	& \ddots 	& \ddots	& 	 \\ 
   	&       	& \ddots      	& F_{0} 	& F_{1} \\ 
	 &  		&  		& F_{2} 	& F_{0} 
\end{array}\right)_{pq \times pq},
\end{displaymath}
\normalsize
\end{notation}
where each of the blocks $F_{i}$ are matrices of size $q \times q$ each.
If $q=1$, then we denote it by $trid_{p}(F_{2}, F_{0}, F_{1})$.
\begin{definition}
The Kronecker product $\otimes$ is an operation on two matrices of 
arbitrary size resulting in a block matrix. Let $A = \left(a_{i,j} 
\right)$ and $B = \left(b_{i,j} \right)$, then by $A \otimes B$ we mean
\small
\begin{displaymath}
A \otimes B = \left(\begin{array}{llll}
a_{11}B & a_{12}B & \dots & a_{1n}B \\ 
\vdots & \ddots &  & \vdots \\ 
\vdots & \dots &  & \vdots \\ 
a_{n1}B & a_{n2}B & \dots & a_{nn}B
 \end{array}\right).
\end{displaymath}
\normalsize
\end{definition}
If the difference operators are scaled by step size $h^{2}$, then 
equation of (\ref{eqnA})
corresponding to the $(i,j,k)^{th}$ grid point is the following:
\begin{align}
 a_{i,j,k}u_{i,j,k} &+ b_{i,j,k}u_{i+1,j,k} + c_{i,j,k}u_{i,j+1,k} + 
d_{i,j,k}u_{i-1,j,k} \nonumber \\
&+ e_{i,j,k}u_{ij-1,k} + f_{i,j,k}u_{i,j,k+1} + g_{i,j,k}u_{i,j,k-1} 
= w_{i,j,k} \label{formula},
\end{align}
where $1\leq i,\,j,\,k\leq n $, and
\begin{align}
 b_{i,j,k} = 0, ~i = n, \quad c_{i,j,k} = 0, ~j = n, \nonumber \\
 f_{i,j,k} = 0, ~k = n, \quad d_{i,j,k} = 0, ~i = 1, \label{const}\\
 e_{i,j,k} = 0, ~j = 1, \quad g_{i,j,k} = 0, ~k = 1. \nonumber
\end{align}
For an anisotropic model problem, we have 
$a_{i,j,k} = 2(l_{1}+l_{2}+l_{3})$, \, $b_{i,j,k} = -l_{1}$, \, $c_{i,j,k} = -l_{2}$, \, $d_{i,j,k} = -l_{1}$,  
$e_{i,j,k} = -l_{2}$, \, $f_{i,j,k} = -l_{3}$, \, $g_{i,j,k} = -l_{3}$, where $w_{i,j,k} = h^{2}r(i,j,k)$. Here the subscript $(i,j,k)$ 
corresponds to the grid location $(ih,jh,kh)$. 
Let $I_{k}$ denote the identity matrix of size $k \times k$. Using the notation
of circulant matrix and Kronecker product, the 
coefficient matrix corresponding to formula (\ref{formula}) is expressed as follows
\begin{align*}
A = Bctrid_{n}(-l_{3}I_{n^{2}}, \widehat{D}, -l_{3}I_{n^{2}}), \quad \widehat{D} = Bctrid_{n}\left(-l_{2}I_{n}, \overline{D}, -l_{2}I_{n}\right), \quad \overline{D} = ctrid_{n}\left(-l_{1}, d, -l_{1}\right).
\end{align*}
We consider now the same problem (\ref{pde}) with the following Dirichlet boundary
condition
\begin{align}\label{diricBdry}
 u(x,y,0) &= 0, \quad u(x,0,z) = 0,  \quad u(0,y,z) = 0. 
\end{align}
To differentiate the Dirichlet problem with that of periodic problem, we shall 
use bold face letters to denote the matrices corresponding to the Dirichlet case.
Using second order finite differences with the Dirichlet
boundary conditions \eqref{diricBdry} above, we obtain the matrix $\mathbf{A}$ 
corresponding to the Dirichlet case as follows
\begin{align*}
 \mathbf{A} = \mathbf{D + L_{1} + L^{T}_{1} + L_{2} + L^{T}_{2} + L_{3} + L^{T}_{3}},
\end{align*}
where
\begin{align*}
\mathbf{L_{3}} = Btrid_{n}(-l_{3}I_{n^{2}}, 0, 0), \quad \mathbf{L_{2}} = I_{n} \otimes Btrid_{n}(-l_{2}I_{n}, 0, 0), \quad \mathbf{L_{1}} = I_{n^{2}} \otimes trid_{n}(-l_{1}, 0, 0).
\end{align*}
For the above model problem the HSSOR preconditioner $\mathbf{B}$ for the Dirichlet 
problem is defined as follows:
\begin{align}
\mathbf{B} &= \left(\mathbf{P} + \mathbf{L_{3}}\right) \left(\mathbf{I} + \mathbf{P^{-1}L^{T}_{3}}\right), \nonumber \\
\mathbf{P} &= \left(\mathbf{T} + \mathbf{L_{2}}\right) \left(\mathbf{I} + \mathbf{T^{-1}L^{T}_{2}}\right), \label{eqnNF}  \\
\mathbf{T} &= \left(\mathbf{M} + \mathbf{L_{1}}\right) \left(\mathbf{I} + \mathbf{M^{-1}L^{T}_{1}}\right), \nonumber
\end{align}
where $\mathbf{M} = diag(\mathbf{A})$.
The HSSOR preconditioner defined above is named RNF(0,0) preconditioner in \cite{kum}. 
Using the notation of circulant matrix and the Kronecker product, the HSSOR 
preconditioner for the periodic boundary condition is now defined as follows
\begin{eqnarray}
B &=& (P + L_{3})(I + P^{-1}L^{T}_{3}),~ P~is~of~size~n^{3} \times n^{3}, \nonumber \\
L_{3} &=& Bcirc_{n}(0,\cdots,0,-l_{3}I_{n^{2}}), \nonumber \\ 
L^{T}_{3} &=& Bcirc_{n}(0,-l_{3}I_{n^{2}},0,\cdots,0), \nonumber \\ 
P &=& I_{n} \otimes P_{0}, ~ P_0~is~of~size~n^{2} \times n^{2}, \nonumber \\ 
P_{0} &=& (\widehat{T} + \widehat{L}_{2})(I + \widehat{T}^{-1}\widehat{L}^{T}_{2}), \nonumber \\ 
\widehat{L}_{2} &=& Bcirc_{n}(0, \cdots, 0, -l_{2}I_{n}), \nonumber \\ 
\widehat{L}^{T}_{2} &=& Bcirc_{n}(0,-l_{2}I_{n}, 0, \cdots, 0), \nonumber \\ 
\widehat{T} &=& I_{n} \otimes T_{0},  ~ T_{0} ~ is ~of ~size~ n \times n, \nonumber  \\ 
T_{0} &=& circ_{n}(m+l_1^2/m,-l_{1},0,\cdots,0,-l_{1}), \nonumber \\ 
m &=& d. 
\end{eqnarray}
It can be proved that $B$ is a circulant matrix.

%% file: anarnf.tex
\section{Fourier analysis of HSSOR}
In this section, we derive the Fourier eigenvalues of the HSSOR
preconditioned matrix. For clarity and simplicity, we restrict our
analysis to the isotropic problem ($l_{1} = l_{2} = l_{3} = 1$),
however, similar analysis holds for the general anisotropic case. In
the following, we outline certain assumptions on which our analysis
will be based. These assumptions are similar to those made in
\cite{ton}, and has been justified their appropriately.

Firstly, our analysis is for the HSSOR preconditioner $B$ and the
coefficient matrix $A$ corresponding to the periodic boundary
conditions. Secondly, Fourier analysis is an exact analysis only for
constant coefficient matrix, which is indeed the case when we have an
isotropic model. According to an argument in \cite{ton}, the extreme
eigenvalues for the periodic and the corresponding Dirichlet problems
are same provided $n = 2n_{d}$. Here $n_{d}+1 = 1/h_{d}$ and $h_{d}$
is the mesh size for the Dirichlet problem.

Eigenvectors of $A$ are found by applying the operator $A$ to
eigenvectors $v^{s,t,r}$.  The $(i,j,k)^{th}$ grid component of
eigenvector $v^{s,t,r}$ is given by
\begin{eqnarray}
  v_{i,j,k}^{s,t,r} = e^{\iota i \theta_{s}} e^{\iota j \phi_{t}} e^{\iota k \xi_{r}}, \label{eqn:eigenvec}
\end{eqnarray}
where $\iota = \sqrt{-1}$, $\theta_{s} = \frac{2\pi}{n+1}s$, $\phi_{t}
= \frac{2\pi}{n+1}t$, and $\xi_{r} = \frac{2\pi}{n+1}r$, for $r,s,t =
1, \cdots, n$.  The eigenvalues $\lambda_{s,t,r}(A)$ of the matrix $A$
is determined by substituting (\ref{eqn:eigenvec}) for $u_{i,j,k}$ in
the left hand side of (\ref{formula}), and it is found to be
\begin{eqnarray}
  \lambda_{s,t,r}(A) = 4\left(l_{1}sin^{2} \frac{\theta_{s}}{2} + l_{2}sin^{2}\frac{\phi_{t}}{2}
    + l_{3}sin^{2}\frac{\xi_{r}}{2}\right). 
\end{eqnarray}
For circulant matrices following results hold.
\begin{lemma}[\cite{dav}] \label{lemmaCir1} Any circulant matrix of
  size $n$ share the same set of eigenvectors.
\end{lemma}
Using lemma (\ref{lemmaCir1}) above, we have the following result.
\begin{lemma} \label{lemmaCir2} Let $S$ and $R$ be two given circulant
  matrices with eigenvalues $\lambda_{s,t,r}(S)$ and
  $\lambda_{s,t,r}(R)$ respectively. Then the eigenvalues of $S+R$ and
  $SR$ corresponding to the $(s,t,r)^{th}$ grid point is given as
  follows: \vspace{-5mm}
\begin{packed_enum}
 \item $\lambda_{s,t,r}(S + R) = \lambda_{s,t,r}(S) + \lambda_{s,t,r}(R)$.
 \item $\lambda_{s,t,r}(SR) = \lambda_{s,t,r}(S) \lambda_{s,t,r}(R)$. 
\end{packed_enum}
\end{lemma}
\begin{proof} It follows easily using lemma (\ref{lemmaCir1})
  above.\end{proof} Using the lemma \ref{lemmaCir2} above, the
eigenvalues $\lambda_{s,t,r}(B^{-1}A)$ of HSSOR preconditioned matrix
is then given by
\begin{eqnarray}
\lambda_{s,t,r}(B^{-1}A) = \frac{\lambda_{s,t,r}(A)}{\lambda_{s,t,r}(B)}, 
\end{eqnarray}
where $\lambda_{s,t,r}(B)$ is given hierarchically as follows:
\begin{align}
  \lambda_{s,t,r}\left(B\right)    & = \left(\lambda_{s,t,r}(P) + \lambda_{s,t,r}(L_{3})\right)\left(1 + \frac{\lambda_{s,t,r}\left(L_{3}^{T}\right)}{\lambda_{s,t,r}\left(P\right)}\right), \nonumber \\
  \lambda_{s,t,r}\left(P\right)    & = \left(\lambda_{s,t,r}(T) + \lambda_{s,t,r}(L_{2})\right)\left(1 + \frac{\lambda_{s,t,r}(L_{2}^{T})}{\lambda_{s,t,r}(T)}\right), \nonumber \\
\lambda_{s,t,r}\left(T\right)      & = \lambda_{s,t,r}\left(M\right) + \lambda_{s,t,r}\left(L_{1}\right) +\lambda_{s,t,r}\left(L_{1}^{T}\right), \nonumber \\
%
\lambda_{s,t,r}\left(L_{1}\right)  & = -l_{1}e^{\iota \theta_{s}}, \quad \lambda_{s,t,r}\left(L_{1}^{T}\right) = -l_{1}e^{-\iota \theta_{s}}\label{eqn:FourierEV}, \nonumber \\
\lambda_{s,t,r}\left(L_{2}\right)  & = -l_{2}e^{\iota \phi_{t}}, \quad \lambda_{s,t,r}\left(L_{2}^{T}\right) = -l_{2}e^{-\iota \phi_{t}}, \nonumber \\
 \lambda_{s,t,r}\left(L_{3}\right) & = -l_{3}e^{\iota \xi_{r}}, \quad \lambda_{s,t,r}\left(L_{3}^{T}\right) = -l_{3}e^{-\iota \xi_{r}}.
\end{align}
where $\lambda_{s,t,r}\left(M\right) = 6$.  The eigenvalues for the
matrices $L_{1},L_{2},L_{3},U_{1},U_{2},U_{3}$, and $M$ were found by
inspection, for instance, if (\ref{formula}) denotes the stencil for
the original matrix $A$, then the stencils (or equations) for the
matrices $L_{1},L_{2},L_{3},L_{1}^{T}, L_{2}^{T},L_{3}^{T}$, and $M$ are
given by
\begin{eqnarray}
stencil~for~M                      & = & mu_{i,j,k}, \nonumber                    \\ 
stencil~for~L_{1}                  & = & -l_{1}u_{i-1,j,k}, \nonumber             \\
stencil~for~L_{1}^{T}              & = & -l_{1}u_{i+1,j,k}, \nonumber             \\
stencil~for~L_{2}                  & = & -l_{2}u_{i,j-1,k},  \label{eqn:stencils} \\
stencil~for~L_{2}^{T}              & = & -l_{2}u_{i,j+1,k}, \nonumber             \\
stencil~for~L_{3}                  & = & -l_{3}u_{i,j,k-1}, \nonumber             \\
stencil~for~L_{3}^{T}              & = & -l_{3}u_{i,j,k+1}. \nonumber
\end{eqnarray}

We recall here that the matrices $M,L_{1},L_{2},L_{3},L_{1}^{T},
L_{2}^{T}$, and $L_{3}^{T}$ are all circulant matrices of same size as
the original coefficient matrix $A$, thus they share the same set of
eigenvectors given by (\ref{eqn:eigenvec}).  After substituting the
eigenvector (\ref{eqn:eigenvec}) in (\ref{eqn:stencils}) for
$u_{i,j,k}$, a straightforward computation gives the required
eigenvalues in (\ref{eqn:FourierEV}). The Fourier eigenvalues for the
HSSOR preconditioner $B$ for the isotropic case, $l1 = l2 = l3 = 1$,
is derived as follows
\begin{align*}
  \lambda_{s,t,r}(B) &= \left(\lambda_{s,t,r}(P) + \lambda_{s,t,r}(L_{3})\right)\left(1 + \frac{\lambda_{s,t,r}\left(L_{3}^{T}\right)}{\lambda_{s,t,r}\left(P\right)}\right), \\
  &= \lambda_{s,t,r}(P) + \lambda_{s,t,r}(L_{3}) + \lambda_{s,t,r}(L_{3}^{T}) + \frac{\lambda_{s,t,r}(L_{3})\lambda_{s,t,r}(L_{3}^{T})}{\lambda_{s,t,r}(P)}, \\
  &= \lambda_{s,t,r}(P) - e^{\iota \xi_{r}} - e^{-\iota \xi_{r}} +
  \frac{e^{\iota \xi_{r}}e^{-\iota \xi_{r}}}{\lambda_{s,t,r}(P)} =
  \lambda_{s,t,r}(P) + \frac{1}{\lambda_{s,t,r}(P)} - 2cos(\xi_{r}),
\end{align*}
where $\lambda_{s,t,r}(P)$ is derived in a similar way as follows
\begin{align*}
  \lambda_{s,t,r}(P) &= \left(\lambda_{s,t,r}(T) + \lambda_{s,t,r}(L_{2})\right)\left(1 + \frac{\lambda_{s,t,r}\left(L_{2}^{T}\right)}{\lambda_{s,t,r}\left(T\right)}\right), \\
  &= \lambda_{s,t,r}(T) + \lambda_{s,t,r}(L_{2}) + \lambda_{s,t,r}(L_{2}^{T}) + \frac{\lambda_{s,t,r}(L_{2}) \lambda_{s,t,r}(L_{2}^{T})}{\lambda_{s,t,r}(T)}, \\
  &= \lambda_{s,t,r}(T) - e^{\iota \phi_{t}} - e^{-\iota \phi_{t}} +
  \frac{e^{\iota \phi_{t}}e^{-\iota \phi_{t}}}{\lambda_{s,t,r}(T)} =
  \lambda_{s,t,r}(T) + \frac{1}{\lambda_{s,t,r}(T)} - 2cos(\phi_{t}),
\end{align*}
where $\lambda_{s,t,r}(T) = 6 - 2cos(\theta_{s})$. 
Due to the periodicity of eigenvalues, i.e.,
\begin{eqnarray*}
  \lambda_{s,t,r}(A)\mid_{(\theta_{s}, \phi_{t}, \xi_{r})} &=& \lambda_{s,t,r}(A)\mid_{(2\pi - \theta_{s}, 2\pi - \phi_{t}, 2\pi - \xi_{r})},\\
  \lambda_{s,t,r}(B)\mid_{(\theta_{s}, \phi_{t}, \xi_{r})} &=& \lambda_{s,t,r}(B)\mid_{(2\pi - \theta_{s}, 2\pi - \phi_{t}, 2\pi - \xi_{r})},
\end{eqnarray*}
we will restrict our domain to $(0, \pi)$ instead of $(0, 2\pi)$.
For any arbitrary matrix $K$, we will use the notation
$\lambda_{min}(K)$ and $\lambda_{\text{max}}(K)$ to denote the minimum
and maximum eigenvalues respectively. When the expression
$\lambda_{s,t,r}(K)$ does not depend on one or more of its arguments
$s$, $t$, or $r$, in such case, we use the dummy argument `*'. For
instance, $\lambda_{s,*,*}(K)$ is an expression independent of the
arguments $t$ and $r$.
\begin{lemma} \label{first} If $\theta_{s}$, $\phi_{t}$, and $\xi_{r}$
  lie in the interval $(0,\pi)$, then following holds \vspace{-5mm}
\begin{packed_enum}
\item $\lambda_{min}(A)=\lambda_{1,1,1}(A)>0$,
\item $\lambda_{min}(T)=\lambda_{1,*,*}(T)>4$,
\item $\lambda_{min}(P)=\lambda_{1,1,*}(P)>9/4$,
\item $\lambda_{min}(B)=\lambda_{1,1,1}(B)>95/36$, 
\item $\lambda_{max}(T)=\lambda_{n/2,*,*}(T)<8$,
\item $\lambda_{max}(P)=\lambda_{n/2,n/2,*}(P)<81/8$,
\item $\lambda_{max}(B)=\lambda_{n/2,n/2,n/2}(B)<7921/648$, 
\item $\lambda_{max}(A)=\lambda_{1,1,1}(A)<12$.
\end{packed_enum}
\end{lemma}
\begin{proof} We observe that $\lambda_{min}(A) =
  4(sin^{2}(\theta_{1}/2) + sin^{2}(\phi_{1}/2) + sin^{2}(\xi_{1}/2))
  = \lambda_{1,1,1}(A) > 0 $.  To prove other parts of the lemma, we
  have $\lambda_{min}(T) = 6 - 2cos(\theta_{1}) =
  \lambda_{1,*,*}(T)>4$.  Now given $x>1$, the expression $x +
  \frac{1}{x}$ increases or decreases according as $x$ increases or
  decreases, consequently, we have
\begin{align*}
  \lambda_{min}(P) &= \lambda_{min}(T) + \frac{1}{\lambda_{min}(T)} - \text{max}(2cos(\phi_{t})), \\
  &= \lambda_{1,*,*}(T) + \frac{1}{\lambda_{1,*,*}(T)} - \text{max}(2cos(\phi_{1})), \\
  &= \lambda_{1,1,*}(P)>9/4.
\end{align*}
Similarly, we have
\begin{align*}
  \lambda_{min}\left(B\right) = \lambda_{min}\left(P\right) +
  \frac{1}{\lambda_{min}\left(P\right)} - \text{max}(2cos(\xi_r)) =
  \lambda_{1,1,1}(B)>95/36.
\end{align*}
On the other hand for the upper bounds, we have $\lambda_{max}(T) = 6
- 2cos(\theta_{n/2}) = \lambda_{n/2,*,*}(T)<8$, and
\begin{align*}
  \lambda_{max}(P) = \lambda_{max}(T) + \frac{1}{\lambda_{max}(T)} -
  \text{min}(2cos(\phi_{t})) = \lambda_{n/2,n/2,*}(P)<81/8.
\end{align*}
Similarly, we have
\begin{align*}
  \lambda_{max}\left(B\right) = \lambda_{max}(P) +
  \frac{1}{\lambda_{max}\left(P\right)} - \text{min}(2cos(\xi_r)) =
  \lambda_{n/2,n/2,n/2}(B)< 7921/648.
\end{align*}
Finally, it is clear that $\lambda_{max}(A)=\lambda_{n/2,n/2,n/2}(A)<12$. 
Hence, the lemma is proved.
\end{proof} 
In the following Lemma, we determine the eigenvalues of $B-A$.
\begin{lemma} \label{L:BminusA}
 The eigenvalues of $B-A$ are given as follows
\begin{align*}
  \lambda_{s,t,r}(B-A) &=  \frac{1}{\lambda_{s,t,r}(T)} +  \frac{1}{\lambda_{s,t,r}(P)}. \\
\end{align*}
\end{lemma}
\begin{proof}
We have
\begin{align*}
  \lambda_{s,t,r}(B) &= \lambda_{s,t,r}(P) + \frac{1}{\lambda_{s,t,r}(P)} - 2 \cos(\xi_r), \\
  &= \lambda_{s,t,r}(T) + \frac{1}{\lambda_{s,t,r}(T)} + \frac{1}{\lambda_{s,t,r}(P)} - 2(\cos(\phi_t)+\cos(\xi_r)), \\
  &= 6-2(\cos(\theta_t)+\cos(\phi_r)+\cos(\xi_r)) + \frac{1}{\lambda_{s,t,r}(T)} + \frac{1}{\lambda_{s,t,r}(P)}, \\
  &= \lambda_{s,t,r}(A) + \frac{1}{\lambda_{s,t,r}(T)} +
  \frac{1}{\lambda_{s,t,r}(P)}.
\end{align*}
The proof is complete.
\end{proof}
Following lemma will be useful in estimating condition number of HSSOR.
\begin{lemma} \label{L:maxminBminusA}
There holds  
\begin{align*}
\lambda_{max}(B-A) &= \lambda_{1,1,1}(B^{-1}A),\\
\lambda_{min}(B-A) &= \lambda_{n/2,n/2,n/2}(B^{-1}A).
\end{align*}
\end{lemma}
\begin{proof}
Using Lemma \ref{L:BminusA}, we have
\begin{align} 
  \lambda_{s,t,r}(B-A) = \lambda_{s,t,r}(B)-\lambda_{s,t,r}(A) =
  \frac{1}{\lambda_{s,t,r}(T)} +
  \frac{1}{\lambda_{s,t,r}(P)}, \label{Exp:lambda_str}
\end{align}
and using Lemma \ref{first} above, we have 
\begin{align*}
  \lambda_{max}(B-A) = \frac{1}{\lambda_{min}(T)} +  \frac{1}{\lambda_{min}(P)} = \frac{1}{\lambda_{1,*,*}(T)} + \frac{1}{\lambda_{1,1,*}(P)}=\lambda_{1,1,1}(B^{-1}A), \\
  \lambda_{min}(B-A) = \frac{1}{\lambda_{max}(T)} +
  \frac{1}{\lambda_{max}(P)} = \frac{1}{\lambda_{n/2,*,*}(T)} +
  \frac{1}{\lambda_{n/2,n/2,*}(P)}=\lambda_{n/2,n/2,n/2}(B^{-1}A).
\end{align*}
\end{proof}
\begin{lemma} \label{L:maxmin}
There holds
\begin{align*}
\lambda_{max}(B^{-1}A) &= \lambda_{1,1,1}(B^{-1}A),\\
\lambda_{min}(B^{-1}A) &= \lambda_{n/2,n/2,n/2}(B^{-1}A).
\end{align*}
\end{lemma}
\begin{proof} 
We have
\begin{align*}
  \lambda_{s,t,r}(B^{-1}A)=
  \frac{\lambda_{s,t,r}(A)}{\lambda_{s,t,r}(B)} =\frac{1}{1 +
    \lambda_{s,t,r}^{-1}(A)(\lambda_{s,t,r}(B-A))}.
\end{align*}
From Lemma \ref{first} above, we have $\lambda_{s,t,r}^{-1}(A)>0$ and
$\lambda_{s,t,r}(B-A)>0$, thus, using Lemma \ref{L:maxminBminusA} we
can write the following
\begin{align}
  \lambda_{min}(B^{-1}A)=\frac{1}{1 + \lambda_{min}^{-1}(A)(\lambda_{max}(B-A))}=\lambda_{1,1,1}(B^{-1}A), \label{minBinvA} \\
  \lambda_{max}(B^{-1}A)=\frac{1}{1 +
    \lambda_{max}^{-1}(A)(\lambda_{min}(B-A))}=\lambda_{n/2,n/2,n/2}(B^{-1}A). \label{maxBinvA}
\end{align}
The proof is complete.
\end{proof}%
The matrix $B^{-1}A$ is symmetric by construction, the following
corollary proves that it is also positive definite.
\begin{corollary}
 The HSSOR preconditioned matrix $B^{-1}A$ is symmetric positive definite.
\end{corollary}
\begin{proof}
 Using Lemma \ref{L:maxmin} and Lemma \ref{first} above, we have 
\begin{align*}
  \lambda_{min}(B^{-1}A) = \lambda_{n/2,n/2,n/2}(B^{-1}A)=
  \lambda_{s,t,r}(A)/ \lambda_{s,t,r}(B)~>~0.
\end{align*}
The proof is complete.
\end{proof}
The following theorem gives a condition number estimate of HSSOR
preconditioned matrix. Let cond(K) for any matrix K denote the
condition number of K. The notation $\approx$ will mean
\emph{`approximately equal to'}.
\begin{theorem}
  Let $h$ tends to zero, then the condition number of the HSSOR
  preconditioned matrix $B^{-1}A$ is given as follows
\begin{align*}
  \text{cond}(B^{-1}A) &\approx \left(\frac{25(5+5\pi^2
      +\pi^4)}{144(3\pi^2(5+5\pi^2+\pi^4)+4\pi^2)}\right)h^{-2}.
\end{align*}
\end{theorem}
\begin{proof}
  From Lemma \ref{L:maxmin} above, the condition number of HSSOR
  preconditioned matrix, cond($B^{-1}A$), is given as follows
\begin{align*}
  \text{cond}(B^{-1}A) =
  \frac{\lambda_{max}(B^{-1}A)}{\lambda_{min}(B^{-1}A)} =
  \frac{\lambda_{n/2,n/2,n/2}(B^{-1}A)}{\lambda_{1,1,1}(B^{-1}A)}.
\end{align*}
Using Lemma \ref{first} and recalling that $\theta_s = 2 \pi s/(n+1) =
(2\pi h)s$, We have $cos(\theta_1) = cos(2 \pi h) \approx 1 - 2 \pi^2
h^2$, consequently, we have
$\lambda_{1,*,*}(T)=6-2cos(\theta_1)\approx 4(1+ \pi^2 h^2)$, thus
using Lemma \ref{L:BminusA} and $cos(\phi_1) = 1 - 2 \pi^2 h^2$, we
have
\begin{align*}
  \lambda_{max}(B-A) &= \lambda_{1,1,1}(B-A),  \\
  &= \frac{1}{\lambda_{1,*,*}(T)} + \frac{1}{\lambda_{1,*,*}(T) + \lambda_{1,*,*}^{-1}(T) - 2cos(\phi_1)}, \\
  &\approx \frac{1}{4(1+\pi^2 h^2)} + \frac{1}{4(1+\pi^2 h^2) + (1/4)(1+\pi^2 h^2)^{-1}-2(1-2\pi^2 h^2)},\\
  &\approx \frac{1}{4(1+\pi^2 h^2)} + \frac{1}{4(1+\pi^2 h^2) + (1/4)(1-\pi^2 h^2)-2(1-2\pi^2 h^2)},\\
  &= \frac{1}{4(1+\pi^2 h^2)} + \frac{1}{2+1/4 + (8-1/4)\pi^2 h^2}, \\
  &= \frac{1}{4(1+\pi^2 h^2)} + \frac{4}{9+31\pi^2 h^2} = \frac{25+47
    \pi^2 h^2}{4(1+\pi^2 h^2)(9+31 \pi^2 h^2)}.
\end{align*}
we have $\lambda_{min}(A)=\lambda_{1,1,1}(A)=4(sin^2(\theta_1 /2 +
\phi_1 /2 + \xi_1 /2)) ~\approx~12\pi^{2}h^2$. From above estimates
and from expression (\ref{minBinvA}) we have
\begin{align*}
  \lambda_{min}(B^{-1}A) &= \lambda_{n/2,n/2,n/2}(B^{-1}A), \\
  &= \frac{1}{1 + \lambda_{min}^{-1}(A)(\lambda_{max}(B-A))}, \\
  &= \frac{1}{1 + (1/12\pi^2 h^2)(25+47\pi^2 h^2)/(4(1+\pi^2 h^2)(9+31\pi^2 h^2))}, \\
  &= \frac{12\pi^2 h^2}{12 \pi^2 h^2 + (25+47\pi^2 h^2)/(4(1+\pi^2 h^2)(9+31 \pi^2 h^2))}, \\
  &= \frac{12 \pi^2 h^2 \cdot 4(1+\pi^2 h^2)(9+31 \pi^2 h^2)}{4 \cdot 12 \pi^2 h^2(1+\pi^2 h^2)(9+31\pi^2 h^2) + 25 + 47\pi^2 h^2},\\
  &= \frac{4 \cdot 9 \cdot 12 \pi^2 h^2 + O(h^4)}{25+O(h^2)} \approx
  \frac{4 \cdot 9 \cdot 12 \pi^2 h^2}{25}.
\end{align*}
We recall that $\theta_s = 2\pi s/(n+1)$. For $s=n/2$, we have
$\theta_{n/2} = n\pi /(n+1) = \pi-\pi /(n+1)= (1-h)\pi$ (since,
$1/(n+1)=h$).  We have $\cos(\theta_{n/2}) \approx~1 -
(\theta_{n/2}^2)/2 = 1 - (1-h)^2 \pi^2/2 \approx
((2-\pi^2)+2\pi^2h)/2$. Due to symmetry, we have
$\cos(\phi_{n/2})=\cos(\xi_{n/2}) = \cos(\theta_{n/2})$.
Consequently, $\lambda_{n/2,*,*}(T)= 6-2cos(\theta_{n/2})=4+\pi^2 -
2\pi^2 h$.  Using Lemma \ref{L:BminusA}, we have
\begin{align*}
  \lambda_{min}(B-A)         &= \lambda_{n/2,n/2,n/2}(B-A), \\
  &= \frac{1}{\lambda_{n/2,*,*}(T)} + \frac{1}{\lambda_{n/2,*,*}(T) + \lambda_{n/2,*,*}^{-1}(T) - 2cos(\phi_{n/2})}, \\
  &\approx \frac{1}{4+\pi^2 - 2\pi^2 h} + \frac{1}{(4+\pi^2 - 2\pi^2 h) + (4+\pi^2 - 2\pi^2 h)^{-1}-((2-\pi^2)+2\pi^2h)/2},\\
  &\approx \frac{1}{4+\pi^2 - 2\pi^2 h} + \frac{1}{(4+\pi^2 - 2\pi^2 h) + (4+\pi^2)^{-1}(1 + 2\pi^2 h/(4+\pi^2))-((2-\pi^2)+2\pi^2h)/2},\\
  &= \frac{1}{4+\pi^2 - 2\pi^2 h} + \frac{1}{4+\pi^2 + 1/(4+\pi^2) - (2-\pi^2)/2 + O(h)}, \\
  &= \frac{2(4+\pi^2)+O(h)}{(4+\pi^2)+2-(2-\pi^2)(4+\pi^2)+O(h)}
  \approx \frac{4+\pi^2}{5+5\pi^2 + \pi^4}.
\end{align*}
From Lemma \ref{first}, we have $\lambda_{max}(A) = \lambda_{n/2,n/2,n/2}(A)=6(1-\cos \theta_{n/2})\approx 6-3(2-\pi^2 + 2\pi^2 h)=3\pi^2 (1-2h)$. Thus, we have
\begin{align*}
  \lambda_{max}(B^{-1}A) &= \frac{1}{1+\lambda_{max}^{-1}(A)(\lambda_{min}(B-A))}, \\
  &\approx \frac{1}{1 + (1/(3\pi^2 (1-2h))((4+\pi^2)/(5+5\pi^2 + \pi^4))}, \\
  &= \frac{3\pi^2 (1-2h)(5+5\pi^2 + \pi^4)}{3\pi^2(1-2h)(5+5\pi^2 + \pi^4) + 4 + \pi^2}, \\
  &= \frac{3\pi^2(5+5\pi^2 + \pi^4) + O(h)}{3\pi^2(5+5\pi^2 +
    \pi^4)+4+\pi^2 +O(h)}.
\end{align*}
The condition number of $B^{-1}A$ is given as follows
\begin{align*}
  \text{cond}(B^{-1}A) &= \frac{3\pi^2(5+5\pi^2 + \pi^4) + O(h)}{3\pi^2(5+5\pi^2 + \pi^4)+4+\pi^2 +O(h)} \times \frac{25}{4\cdot 9 \cdot 12\pi^2 h^2}, \\
  &\approx \left(\frac{25(5+5\pi^2
      +\pi^4)}{144(3\pi^2(5+5\pi^2+\pi^4)+4\pi^2)}\right)h^{-2}
  \approx (0.006)h^{-2}.
\end{align*}
The proof is complete.
\end{proof}

%% file: gmgrnf.tex
\section{Two grid scheme}
In classical AMG, a set of coarse grid unknowns is selected and the
matrix entries are used to build interpolation rules that define the
prolongation matrix P, and the coarse grid matrix $A_c$ is computed
from the following Galerkin formula
\begin{eqnarray} \label{E:galerkin} A_c = P^{T}AP.
\end{eqnarray}
In contrast to the classical AMG approach, in aggregation based
multigrid, first a set of aggregates $G_{i}$ are defined. Let $N_c$
be the number of such aggregates, then the interpolation matrix $P$ is
defined as follows
\begin{equation*} \label{interp}
P_{ij} =
\begin{cases}
  1, &\text{if $i \in G_{j}$,}\\
  0, &\text{otherwise,}\\
\end{cases}
\end{equation*}
Here, $1 \le i \le N, \, 1 \le j \le N_c$, $N$ being the size of the
original coefficient matrix $A$. Further, we assume that the
aggregates $G_{i}$ are such that
\begin{equation} \label{aggr}
G_{i}\cap G_{j}=\phi,~ \text{for}~ i \neq j~ \text{and}~
\cup_iG_{i}= [1,N]
\end{equation}
Here $[1,\, N]$ denotes the set of integers from $1$ to $N$.
 Notice that the matrix $P$ defined above is an $N \times N_c$ matrix, but since it
has only one non-zero entry (which are ``one'') per row, the matrix
can be defined by a single array containing the indices of the non-zero entries. 
The coarse grid matrix $A_c$ may be computed as follows
\[
(A_c)_{ij} = \sum_{k \in G_i} \sum_{l \in G_j} a_{kl}
\]
where $1 \le i, \ j \le N_c$, and $a_{kl}$ is the $(k,l)th$ entry of $A$.

Numerous aggregation schemes have been proposed in the literature, but
in this paper we consider the aggregation scheme based on graph
matching as follows \vspace{-5mm}
\begin{description}
\item {\bf Aggregation based on graph matching:} Several graph
  partitioning methods exists, notably, in software form
  \cite{kar3}. Aggregation for AMG is created by calling a graph
  partitioner with required number of aggregates as an input. The
  subgraph being partitioned are the aggregates $G_i$. For instance,
  in this paper we use this approach by calling the METIS graph
  partitioning routine, namely, METIS\_PartGraphKway with the graph of
  the matrix and number of partitions as input parameters. The
  partitioning information is obtained in the output argument
  ``part". The part array maps a given node to its partition, i.e.,
  part($i$) = $j$ means that the node $i$ is mapped to the $jth$
  partition. In fact, the part array essentially determines the
  interpolation operator $P$. For instance, we observe that the
  ``part'' array is a discrete many to one map. Thus, the $i$th
  aggregate $G_{i}=\text{part}^{-1}(i)$, where
  \[ 
   \text{part}^{-1}(i) = \{ \, j \in [1,\, N] \enspace \mid \enspace 
\text{part}(j)=i\,\}
  \]
  Such graph matching techniques for AMG coarsening were also explored
  in \cite{kim,ras}. For notational convenience, the method
  introduced in this paper is called GMG (Graph matching based
  aggregation MultiGrid).
\end{description}
\vspace{-5mm} Let $S$ denote the matrix which acts as a smoother in
GMG method. The usual choice of $S$ is a Gauss-Siedel or SSOR
preconditioner \cite{saad96}. However, in this paper we choose HSSOR
as a smoother and compare it with SSOR. 

Let $M=PA_cP^{T}$ denote the coarse grid operator {\em interpolated}
to fine grid, then the two-grid preconditioner without post-smoothing
is defined as follows
\begin{eqnarray} \label{twogrid} B = (S^{-1} + M^{-1} -
  M^{-1}AS^{-1})^{-1}.
\end{eqnarray}
We notice that $M^{-1} \approx PA_{c}^{-1}P^{T}$, thus, an equation of
the form $Mx=y$ is solved by first restricting $y$ to $y_c=P^T y$,
then solving with the coarse matrix $A_c$ the following linear system:
$A_cx_c = y_c$.  Finally, prolongating the coarse grid solution $x_c$
to $x=Px_c$. Following diagram illustrates the two-grid hierarchy.
\begin{displaymath}
  \xymatrix{
    \cdots \bullet-\bullet-\bullet-\bullet\cdots \ar[d]_{\text{Restrict $y$ to $y_c:=P^{T}y$}} & \cdots \bullet-\bullet-\bullet-\bullet\cdots  \\
    \cdots \bullet-\bullet \cdots \ar[r]_{\text{Solve:$A_cx_c$ = $y_c$}} & \cdots \bullet-\bullet \cdots \ar[u]_{\text{Prolongate $x_c$ to $x:=P x_c$}}} 
\end{displaymath}

%% file: numrnf.tex
\section{Numerical experiments}
All the numerical experiments were performed in MATLAB with double
precision accuracy on Intel core i7 (720QM) with 6 GB RAM. The AMG
method introduced in this paper, namely, GMG, is written in MATLAB.
For GMG, the iterative accelerator used is GMRES available at
\cite{saad_soft}, the code was changed such that the stopping is based
on the decrease of the 2-norm of the relative residual. For both
GMRES, the maximum number of iterations allowed is 500, and the inner
subspace dimension is 30. The stopping criteria is the decrease of the
relative residual below $10^{-10}$, i.e., when
\[
\frac{\|b-Ax_k\|}{\|b\|} < 10^{-10}.
\]
Here $b$ is the right hand side and $x_k$ is an approximation to the
solution at the $k$th step.

\begin{table}
  \caption{Notations used in tables of numerical experiments }
  \label{not}
  \begin{center}
    \begin{tabular}{ll}
      \toprule
      \addlinespace
      Notations & Meaning                                                       \\
      \addlinespace
      \midrule
      \addlinespace
      $h$       & Discretization step                                           \\
      $N$       & Size of the original matrix                                   \\
      $N_c$     & Size of the coarse grid matrix                                \\ 
      its       & Iteration count                                               \\
      time      & Total CPU time (setup plus solve) in seconds                  \\
      $cf$      & $(N)^{1/3}/(N_c)^{1/3}$                                       \\
      ME        & Memory allocation problem                                     \\
      NA        & Not applicable                                                \\
      NC        & Not converged within 500 iterations                                     \\
      SSOR      & Symmetric successive over-relaxation method ($\omega=1$)\cite{saad96}       \\
      BSSOR     & Block symmetric successive over-relaxation method ($\omega=1$)\cite{saad96} \\
      HSSOR     & Hierarchical SSOR                                        \\
      GMG-HS    & Graph based matching for AMG, smoother HSSOR                 \\
      GMG-SS    & Graph based matching for AMG, smoother SSOR           \\
      \addlinespace
      \bottomrule
    \end{tabular}
  \end{center}
\end{table}

\subsection{Test cases}
\begin{description}
\item {\bf Diffusion:} Our primary test case is the following
  diffusion Equation. We use the notation DC to indicate that the
  problems are discontinuous. We consider a test case as follows
  \begin{eqnarray}\label{pde}
    -\text{div}(\kappa(x)\nabla u)&=&f~\text{in}~\Omega, \notag \\
    u&=&0~\text{on}~\partial\Omega_D, \\
    \frac{\partial u}{\partial n}&=&0~\text{on}~\partial\Omega_N, \notag
  \end{eqnarray}
  \begin{description}
  \item {\bf DC1, 2D case:} The tensor $\kappa$ is isotropic and
    discontinuous. The domain contains many zones of high permeability
    that are isolated from each other. Let $[x]$ denote the integer
    value of $x$. For two-dimensional case, we define $\kappa(x)$ as
    follows:
    \begin{eqnarray*}
      \kappa(x)=\left\{
        \begin{array}{ll}
          10^3\ast ([10\ast x_2]+1),~ \text{if} \hspace{0.2cm}[10\ast x_i]\equiv0
          \hspace{0.12cm}~(mod~2),~i=1, 2,\\
          1,~ \text{otherwise}.
        \end{array}
      \right.
    \end{eqnarray*}
    The velocity field $\mathbf{a}$ is kept zero.  We consider a $n
    \times n$ uniform grid where $n$ is the number of discrete points
    along each spatial directions.
  \item {\bf DC1, 3D case:} For three-dimensional case, $\kappa(x)$ is
    defined as follows:
    \begin{eqnarray*}
      \kappa(x)=\left\{
        \begin{array}{ll}
          10^3\ast ([10\ast x_2]+1),\hspace{0.3cm} \text{if} \hspace{0.2cm}[10\ast x_i]\equiv 0
          \hspace{0.12cm} (mod~2) \hspace{0.12cm}, \hspace{0.12cm} i=1, 2, 3,\\
          1,~\text{otherwise}.
        \end{array}
      \right.
    \end{eqnarray*}
    Here again, the velocity field $\mathbf{a}$ is kept zero. We
    consider a $n \times n \times n$ uniform grid.

  \item {\bf Isotropic problem:} i.e., we have $l1 = l2 = l3 = 1$ and
    $d=6$ for the model problem (\ref{pde}). More precisely, the
    matrix is given as follows
    \begin{align*}
      A = Btrid_{n}\left(-l_{3}I_{n^{2}}, \widehat{D},
        -l_{3}I_{n^{2}}\right), \quad \widehat{D} =
      Btrid_{n}\left(-l_{2}I_{n}, \overline{D}, -l_{2}I_{n}\right),
      \quad \overline{D} = trid_{n}\left(-l_{1}, d, -l_{1}\right).
    \end{align*}
    For the 2D case, we have $l1=l2=1$, and $d=4$, the matrix after
    discretization is given as follows
    \begin{align*}
      A = Btrid_{n}\left(-l_{2}I_{n}, \overline{D},
        -l_{2}I_{n}\right), \quad \overline{D} = trid_{n}\left(-l_{1},
        d, -l_{1}\right).
    \end{align*}
  \end{description}
\end{description}
\subsection*{Comments on numerical experiments}
In Table \ref{T:isotropic}, we compare the two-grid methods, ILU(0),
HSSOR, SSOR, and BSSOR.  In two grid methods, GMG-SS has point SSOR
smoother while GMG-HS has HSSOR as a smoother. For the 2D case and for
$1/h \ge 400$, the ILU(0), SSOR, HSSOR, and BSSOR does not converge
within 500 iterations, while all of them converge for 3D problem
except for large size problem where insufficient memory error
occurs. In particular, for block SSOR method, we need to store the LU
factors corresponding to the 2D blocks. This is the reason why we
refrain from using BSSOR as a smoother in the two-grid scheme. The
iteration count as well as CPU time for HSSOR is smaller compared to
ILU(0) and the difference between iteration count and CPU time
increases with the size of the problem.

In Table \ref{T:isotropic}, as expected, the two-grid methods show
mesh independent convergence, and the CPU time and iteration count is
much less than that of the stand alone methods. Comparing the two-grid
versions GMG-SS and GMG-HS, we find that GMG-HS is an improvement over
GMG-SS, particularly, for the 2D isotropic problem.

In Table \ref{T:jumps}, we show experimental results for a
discontinuous DC1 problem both for 2D and 3D problem. This problem is
difficult compared to isotropic case, we had to keep a smaller $cf$
value of 3. For $cf=4.5$, the two-grid methods did not converge within
500.  Notably, neither of the stand-alone methods converged within 500
iterations. However, for $cf=3$, the two-grid method shows mesh
independent convergence with GMG-HS taking relatively less iterations,
and takes less CPU time compared to GMG-SS.
%
%
\small
\begin{table}
  \caption{Numerical results for isotropic model problem with $cf=4.5$ using GMRES(30), maximum number 
    of iterations allowed is 500}
  \label{T:isotropic}
  \begin{center}
    \begin{tabular}{llllllllllllll}
      \toprule
      \addlinespace
      matrix & $1/h$                      & \multicolumn{2}{c}{GMG-HS} & \multicolumn{2}{c}{GMG-SS}
             & \multicolumn{2}{c}{ILU(0)} & \multicolumn{2}{c}{HSSOR}  & \multicolumn{2}{c}{SSOR} & \multicolumn{2}{c}{BSSOR}                                             \\
      \addlinespace
      \cmidrule(lr){3-4} \cmidrule(lr){5-6} \cmidrule(lr){7-8}
      \cmidrule(lr){9-10} \cmidrule(lr){11-12} \cmidrule(lr){13-14}
      \addlinespace
             &                            & its                        & time                      
             & its                        & time                       & its                      & time & its   & time & its   & time & its   & time                     \\
      \addlinespace
      \midrule
      \addlinespace
             & 400                        & 39                         & 6.2                      & 46   & 6.8   & NC   & NA    & NC   & NA    & NC  & NA    & NC  & NA   \\
    2D       & 800                        & 39                         & 28.9                     & 47   & 37.1  & NC   & NA    & NC   & NA    & NC  & NA    & NC  & NA   \\
             & 1000                       & 42                         & 47.5                     & 50   & 60.9  & NC   & NA    & NC   & NA    & NC  & NA    & NC  & NA   \\
      \addlinespace
             & 40                         & 23                         & 1.3                      & 21   & 3.7   & 55   & 1.8   & 42   & 1.6   & 68  & 2.7   & 127 & 83.4 \\
   3D        & 80                         & 25                         & 24.5                     & 22   & 138.7 & 129  & 61.9  & 89   & 38.7  & 157 & 76.6  & ME  & NA   \\
             & 100                        & 26                         & 72.3                     & ME   & NA    & 147  & 138.3 & 113  & 106.4 & 185 & 180.7 & ME  & NA   \\
      \addlinespace
      \bottomrule
    \end{tabular}
  \end{center}
\end{table}
\normalsize
%
%
\small
\begin{table}
  \caption{Numerical results for DC1 problem with $cf=3$ using GMRES(30). Maximum number 
of iterations allowed is 500}
  \label{T:jumps}
  \begin{center}
    \begin{tabular}{llllllllllllll}
     \toprule
\addlinespace
      matrix & $1/h$                      & \multicolumn{2}{c}{GMG-HS} & \multicolumn{2}{c}{GMG-SS}
             & \multicolumn{2}{c}{ILU(0)} & \multicolumn{2}{c}{HSSOR}  & \multicolumn{2}{c}{SSOR} & \multicolumn{2}{c}{BSSOR}\\
\addlinespace
      \cmidrule(lr){3-4} \cmidrule(lr){5-6} \cmidrule(lr){7-8}
\cmidrule(lr){9-10} \cmidrule(lr){11-12} \cmidrule(lr){13-14}
\addlinespace
             &                            & its                        & time                      
             & its                        & time                       & its                      & time & its   & time & its   & time & its   & time\\
\addlinespace
\midrule
\addlinespace
             & 400                        & 29                         & 5.7                      & 35   & 10.5  & NC   & NA    & NC   & NA    & NC  & NA    & NC  & NA \\
     2D      & 800                        & 29                         & 33.4                     & 34   & 49.5  & NC   & NA    & NC   & NA    & NC  & NA    & NC  & NA \\
             & 1000                       & 29                         & 58.4                     & 35   & 85.0  & NC   & NA    & NC   & NA    & NC  & NA    & NC  & NA\\
\addlinespace
             & 40                         & 247                        & 18.5                     & 300  & 44.0  & 475  & 14.2  & NC   & NA    & NC  & NA    & NC  & NA \\
    3D       & 80                         & 237                        & 337.5                    & 281  & 895.9 & NC   & NA    & NC   & NA    & NC  & NA    & NC  & NA \\
             & 100                        & ME                         & NA                       & ME   & NA    & NC   & NA    & NC   & NA    & NC  & NA    & NC  & NA \\
\addlinespace
      \bottomrule
    \end{tabular}
  \end{center}
\end{table}
\normalsize

%% file: conrnf.tex
\section{Conclusion}
In this paper, we have obtained a condition number estimate for a
hierarchical SSOR method. The estimate facilitates comparison with the
condition number of ILU(0) obtained in \cite{don}. Numerical
experiments shows that the HSSOR is faster compared to ILU(0), SSOR,
and BSSOR as a stand alone preconditioner. Moreover, for a two-grid
method, we show that HSSOR is an efficient smoother and thus could
replace the widely used Gauss-Siedel or SSOR smoother.